\documentclass{elsarticle} 

 \usepackage{subcaption}

\usepackage[mathlines]{lineno}
\usepackage{amsmath, amsfonts, amssymb, amsthm}
\usepackage{natbib}
\usepackage{hyperref}
\usepackage{etoolbox}
\usepackage{times}
\usepackage{xcolor}
\usepackage{enumitem}
\usepackage{url} 
\usepackage{graphicx}
\newcommand*\linenomathpatch[1]{%
  \cspreto{#1}{\linenomath}%
  \cspreto{#1*}{\linenomath}%
  \csappto{end#1}{\endlinenomath}%
  \csappto{end#1*}{\endlinenomath}%
}
\newcommand*\linenomathpatchAMS[1]{%
  \cspreto{#1}{\linenomathAMS}%
  \cspreto{#1*}{\linenomathAMS}%
  \csappto{end#1}{\endlinenomath}%
  \csappto{end#1*}{\endlinenomath}%
}

\expandafter\ifx\linenomath\linenomathWithnumbers
  \let\linenomathAMS\linenomathWithnumbers
   \patchcmd\linenomathAMS{\advance\postdisplaypenalty\linenopenalty}{}{}{}
 \else
   \let\linenomathAMS\linenomathNonumbers
 \fi

\linenomathpatch{equation}
\linenomathpatchAMS{gather}
\linenomathpatchAMS{multline}
\linenomathpatchAMS{align}
\linenomathpatchAMS{alignat}
\linenomathpatchAMS{flalign}

\makeatletter
\patchcmd{\mmeasure@}{\measuring@true}{
  \measuring@true
  \ifnum-\linenopenaltypar>\interdisplaylinepenalty
    \advance\interdisplaylinepenalty-\linenopenalty
  \fi
  }{}{}
\makeatother

\newcommand{\X}{\mathcal{X}}
\newcommand{\Hs}{\mathcal{H}}

\newcommand{\I}{\mathbb{I}}
\newcommand{\CalL}{\mathbb{L}}

\newcommand{\N}{\mathcal{N}}
\newcommand{\A}{\mathbb{A}}
\newcommand{\Real}{\mathbb{R}}
\newcommand{\EXP}{\mathbb{E}}

\newcommand{\norm}[1]{\left\lVert#1\right\rVert}
\newcommand{\innerpro}[1]{\left\langle#1\right\rangle}

\makeatletter
\newcommand{\dlmf}[1]{%
\cite[%
  \def\nextitem{\def\nextitem{, }}%
  \@for \el:=#1\do{\nextitem\href{http://dlmf.nist.gov/\el}{(\el)}}%
]{Olver:10}%
}
\makeatother

 \newtheorem{lemma}{Lemma}
 \newtheorem{theorem}{Theorem}
 \newtheorem{remark}{Remark}
 \newtheorem{corollary}{Corollary}
 \newtheorem{assumption}{Assumption}
 \newtheorem{definition}{Definition}

\begin{document}
\begin{frontmatter}
\title{On regularized polynomial functional regression}

\author[1]{Markus Holzleitner\corref{cor1}}
\ead{holzleitner@dima.unige.it}

\author[2]{Sergei V. Pereverzyev}
\ead{sergei.pereverzyev@oeaw.ac.at}

\affiliation[1]{organisation={MaLGa Center, Department of Mathematics, University of Genoa}, adressline={Via Dodecaneso 35}, postcode={16146}, city={Genoa}, country={Italy}}

\affiliation[2]{organisation={Johann Radon Institute for Computational and Applied Mathematics, Austrian Academy of Sciences}, adressline={Altenberger Straße 69}, postcode={A-4040}, city={Linz}, country={Austria}}

\cortext[cor1]{Corresponding author}

\begin{abstract}
This article offers a comprehensive treatment of polynomial functional regression, culminating in the establishment of a novel finite sample bound. This bound encompasses various aspects, including general smoothness conditions, capacity conditions, and regularization techniques. In doing so, it extends and generalizes several findings from the context of linear functional regression as well. We also provide numerical evidence that using higher order polynomial terms can lead to an improved performance.
\end{abstract}

\date{}
\end{frontmatter}


\section{Introduction}
Functional data are frequently used and cover a wide range of applications, e.g. in engineering, econometrics, health-care, medical images, just to name a few. These data are usually given in the form of functions, time series, shapes, or more general objects. Specific examples include traffic flow data, weather data, functional magnetic imaging resonance data and many more. However, their intrinsic infinite dimensionality poses major challenges in the development and analysis of efficient learning algorithms. While the term "functional data analysis" was coined by \cite{ramsay1982data,ramsay1991some}, it is nowadays a very popular research area, and for comprehensive overview articles on the current state of the field, in terms of methods, theory and applications, we refer e.g. to \cite{ramsay2002applied, wang2016functional, kokoszka2017introduction, reiss2017methods}.

In this paper we focus on functional regression with scalar response, and the arguably most well studied method in this context assumes a linear relation between inputs and outputs, so that the responses are described by a linear functional of the (functional) inputs and some additional noise term.
Two main approaches currently exist in the literature to tackle this problem. One is functional principal component analysis \cite{cai2006prediction, hall2007methodology, yao2005functional}, which assumes that the slope can be described by a basis consisting of the leading functional principal components of the explanatory variable. 

The second one assumes that the slope lies in a reproducing kernel Hilbert space (RKHS). This approach allows for using techniques from standard works on kernel regression (see e.g. \cite{caponnetto2007optimal, guo2017learning, lu2020balancing, lin2017distributed}) and is therefore very popular, see e.g. \cite{tong2018analysis, tong2021distributed,yuan2010reproducing} and references therein. Also more sophisticated results have recently been developed in this setting, that allow e.g. for dealing with online learning \cite{chen2022online, guo2023capacity} and distributed learning \cite{tong2021distributed,tong2018analysis,liu2022statistical}. It is also worth mentioning several works in the more general context of linear function on function regression (where the responses are also functions), see e.g. \cite{tong2022non,mollenhauer2022learning,benatia2017functional}.
The goal is always to estimate the slope and intercept of the underlying linear model, based on a regularized linear predictor constructed from training data, and provide associated non-asymptotic convergence guarantees.
All these results have in common, that they lack generality in the sense that they are not simultaneously studying the interplay of smoothness, capacity conditions and general regularization scheme, with one notable exception being \cite{lin2020optimal} in the related setting of least squares regression over Hilbert spaces.
 Just like in standard polynomial regression, where the linear model is embedded into a polynomial version, it is also natural to consider such scenario in the functional setting, in order to get better predictors. However, this field of research is still sparsely investigated, and we are only aware of \cite{yao2010functional}, who propose an extension with quadratic functional regression as the most prominent example. Note also the work by \cite{horvath2013test}, who consider a hypothesis test to find out, whether a linear model is sufficient.
To the best of our knowledge, and similar to the linear case, a systematic and general treatment of functional polynomial regression is still missing. The main aim of our work is to fill this gap and thus our contributions can be summarized as follows:

\begin{itemize}
\item Introduce polynomial functional regression in a rather general setting and derive finite sample bounds. This allows for inclusion of general smoothness conditions, capacity conditions and regularization techniques as done e.g. in the work by \cite{guo2017learning,lin2017distributed, lu2020balancing,pereverzyev2022introduction} for kernel regression in the standard supervised setting. Our work does not even assume an RKHS setting, and, also in this regard, is therefore more general than most of the previous works dealing with the linear case.
\item Provide a systematic way on how the resulting predictor model can be evaluated explicitly in case of (one parameter) Tikhonov regularization. Surprisingly, to compute the required coefficients, we found out that it is sufficient to only evaluate the linear case. 
\item Provide a numerical toy example that shows the possible advantage of using higher order functional regression.
\end{itemize}

Our work will now be structured as follows: the next section introduces the required notions and proves our main result, a finite sample bound for polynomial regression. Then we will discuss an algorithmic realization of regularized polynomial functional regression in the case of Tikhonov regularization and finally, introduce a toy example demonstrating the advantage of using higher order polynomials. 

\section{Setting and main result}

\subsection{Overall setting and assumptions} \label{subsec:notation}

Let $\mathbb{I} \subset \mathbb{R}^d$ and consider the associated space $L^2\left(\mathbb{I}\right)$ consisting of square integrable functions with respect to the Lebesgue measure $\mu$, so that
\begin{align*}
\norm{u}^2_{L^2(\mathbb{I})}=\int_{\mathbb{I}} |u(t)|^2 d \mu(t).
\end{align*}
Moreover, let $L^2(\Omega, \mathbb{P})$ be a space of random variables $Y=Y(\omega)$ defined on a probability space $(\Omega, \mathcal{F},\mathbb{P})$, $\omega \in \Omega$, with bounded second moments, 
so that
\begin{align*}
\norm{Y}^2_{L^2(\Omega, \mathbb{P})}:= \mathbb{E}|Y|^2 = \int_{\Omega} |Y(\omega)|^2 d\mathbb{P}(\omega).
\end{align*}
Consider also the tensor product $L^2(\Omega, \mathbb{P}) \otimes L^2\left(\mathbb{I}\right)$, 
which is nothing but a collection of random variables $X(\omega,s)$ indexed by points $s \in \mathbb{I}$ and having bounded second moments in the following sense:
\begin{align*} 
\norm{X}_{\mathbb{P}, \mu}^2:=\mathbb{E} \norm{X(\omega,\cdot )}^2_{L^2(\mathbb{I})}.
\end{align*}
The inner products in the considered Hilbert spaces $\Hs$ will always be denoted by $\innerpro{.,.}_{\Hs}$, and the space is indicated by a subscript.

Functional data consist of random i.i.d. samples of functions $X_1(s),...,X_N(s)$, that can be seen as a realization of a stochastic process $X(\omega,s) \in L^2(\Omega, \mathbb{P}) \otimes L^2\left(\mathbb{I}\right)$. Now let us discuss the setting of polynomial functional regression (PFR): Let $Y \in L^2(\Omega, \mathbb{P}) $ be a scalar response, and $X \in L^2(\Omega, \mathbb{P}) \otimes L^2\left(\mathbb{I}\right)$ a corresponding functional predictor. We make the following assumption on $X$ (as imposed in a similar way e.g. in \cite{yuan2010reproducing,tong2018analysis}):
\begin{assumption} \label{ass:unif}
\begin{align*}  
\sup_{\omega \in \Omega} \norm{X(\omega,\cdot )}_{L^2(\I)} \le \kappa.
\end{align*}
\end{assumption}
In PFR one aims at minimizing the expected prediction risk: 
\begin{align} \label{eq:pftfr}
    \mathcal{E}(U_p(X))=\mathbb{E} \left( |Y(\omega )-U_p(X(\omega,\cdot )) |^2 \right) \to \min,
\end{align}
where $U_p(X(\omega,\cdot ))$ is a polynomial regression of order $p$:
\begin{align*}
U_p(X(\omega, \cdot))=u_0+\sum_{l=1}^p \int_{\I^l} u_l(s_1,...,s_l) \prod_{j=1}^l  X(\omega,s_j) d\mu(s_j).
\end{align*}
Here $u_l \in L^2_l$, where
\begin{align*}
L^2_l=\underbrace{L^2(\I) \otimes \cdots \otimes L^2(\I)}_{l\text{ -times}}.
\end{align*}

To proceed and formalize the setting further, consider the operator
\begin{align*}
A_0: \Real \to L^2(\Omega, \mathbb{P}) 
\end{align*}
assigning to any $u_0 \in \Real$ the corresponding constant random variable. Moreover, consider $A_l: L^2_l \to  L^2(\Omega, \mathbb{P}) $, such that
\begin{align} \label{eq:aldef}
(A_l u)(\omega)= \int_{\I^l} u_l(s_1,...,s_l) \prod_{j=1}^l  X(\omega,s_j) d\mu(s_j).
\end{align}
Let, also,  $\CalL^2=\bigoplus_{l=0}^p L^2_l$ be a direct sum of spaces $L^2_l$ consisting of finite sequences $u=(u_0,...,u_p)$, 
$u_0 \in L^2_0=\Real$, $u_l \in L^2_l$, $l=1,2,...,p$, equipped with the norm $\norm{u}_{\CalL^2}^2=\sum_{l=0}^p \norm{u_l}_{L^2_l}^2$, and consider the bounded linear operator (which is also a Hilbert-Schmidt one, as will be seen from Lemma \ref{lem:hs_bounds}) $\A:\CalL^2 \to L^2(\Omega, \mathbb{P})$, given by
\begin{align} \label{eq:CalApdef}
\A u=(A_0,A_1,...,A_p) \circ (u_0,u_1,...,u_p)=\sum_{l=0}^p
A_l u_l. 
\end{align}

Observe that 
\begin{align*}
A_l^*: L^2(\Omega, \mathbb{P}) \to L^2_l, (A_l^* Z)(s_1,...s_l)=\int_{\Omega} Z(\omega) \prod_{i=1}^l  X(\omega,s_i) d\mathbb{P}(\omega),
\end{align*}
because
          \begin{align*}
&\innerpro{A_l^* Z, u}_{L^2_l}= \int_{\I^l} \int_{\Omega} Z(\omega) \prod_{i=1}^l  X(\omega,s_i) d\mathbb{P}(\omega) u(s_1,...,s_l)   \prod_{i=1}^l d\mu(s_i)\\
&=\int_{\Omega} Z(\omega) \left(\int_{\I^l} u(s_1,...,s_l) \prod_{i=1}^l  X(\omega,s_i) d\mu(s_i) \right)  d\mathbb{P} (\omega)=\innerpro{Z, A_l u}_{L^2(\Omega, \mathbb{P})},
\end{align*}
and therefore, $\A^* \A$ is a $(p+1) \times (p+1)$ matrix of the operators 
\begin{align*}
\A^* \A= \left\{ A_k^* A_l: L^2_l \to L^2_k, k,l=0,1,...,p \right\},
\end{align*}
where $A_0^{*} A_0u_0 =u_0$ and
\begin{align} 
A_0^{*} A_lu&=\int_{\Omega} \int_{\I^l} u(s_1,...,s_l) \prod_{i=1}^l X(\omega,s_i) d\mu(s_i) d \mathbb{P} (\omega), \nonumber\\
A_k^{*} A_lu(s_1,...,s_k)&=\int_{\Omega} \prod_{j=1}^k X(\omega,s_j) \int_{\I^l} u(\tilde{s}_1,...,\tilde{s}_l) \prod_{i=1}^l X(\omega,\tilde{s}_i) d\mu(\tilde{s}_i)  d \mathbb{P} (\omega),  \nonumber\\
k,l&=1,...,p.  \nonumber 
\end{align}

Equipped with this notation, we have that
$U_p(X(\omega, \cdot))=\A u$, such that \eqref{eq:pftfr} is reduced to the least square solution of the equation $\A u=Y$, because $\mathcal{E}(U_p(X))=\norm{Y-\A u}^2_{L^2(\Omega, \mathbb{P})}$. Let us also use the following standard assumption:
\begin{assumption} \label{eq:ass:projection}
The projection $\mathcal{P} Y$ of $Y$ on the closure of the range of $\A$ is such that $\mathcal{P} Y \in \text{Range}(\A)$.
\end{assumption}
It is well known (see e.g. \cite{lu2013regularization}[Proposition 2.1.]), that under Assumption \ref{eq:ass:projection} the minimizer $u = u^+=(u_0^+,...,u_p^+)$ of \eqref{eq:pftfr} solves $\A^* \A u=\A^* Y$.

For the sake of analysis let us also adopt the following response noise model:
\begin{assumption} \label{as:noise_model}
\begin{align} \label{eq:noise_model}
Y= \A u^+ + \varepsilon,
\end{align}
 where a noise variable $\varepsilon: \Omega \to \Real$ is independent from $X$, $\mathbb{E}(\varepsilon)=0$, and for some $\sigma>0$ it should satisfy either the condition
\begin{align} \label{eq:noise_1}
\mathbb{E}(|\varepsilon(\omega)|^2) \le \sigma^2,
\end{align}
 or obey, for any integer $\tilde{m} \geq 2$ and some $M>0$ , a slightly stronger moment condition, which is also standard in the literature, see e.g. \cite{tong2021distributed}, 
\begin{align} \label{eq:noise_2}
\mathbb{E}(|\varepsilon(\omega)|^{\tilde{m}}) \le \frac12 \sigma^2 \tilde{m}! M^{\tilde{m}-2}.
\end{align}
\end{assumption}
However, the involved operators are inaccessible, because we do not know $\mathbb{P}$. Thus, we want to approximate them by using training data $(Y_i, X_i(\cdot))$, $i=1,...,N$, consisting of $N$ independent samples
of the response and the functional predictor $(Y(\omega), X(\omega, \cdot))$, so that
\begin{align*}
Y_i=\A_i u^+ +\varepsilon_i,
\end{align*}
where $\A_i$ is defined in the same way as $\A$ by the replacement of $X(\omega, \cdot)$
in the formulas \eqref{eq:aldef} and \eqref{eq:CalApdef} with $X_i(\cdot)$, and $\varepsilon_i$ is a sample from the noise variable introduced in Assumption \ref{as:noise_model}.

Moreover, $u^+$ does not depend continuously on the initial datum, such that we need to employ a regularization.

Let us start by considering Tikhonov regularization, i.e. for $\lambda>0$ we want to find the minimizer $u_{\lambda}$ of the regularized PFR
\begin{align} \label{eq:tikpftfr}
\norm{Y-\A u}^2_{L^2(\Omega, \mathbb{P})} + \lambda \norm{u}^2_{\mathbb{L}^2} \to \text{min},
\end{align} 
which solves the equation $\lambda u + \A^* \A u= \A^* Y$ and can be approximated by the solution $u_{\lambda}^N$ of 
\begin{align} \label{eq:tikpftfremp}
\lambda u + [\A^* \A]_N u= [\A^* Y]_N ,
\end{align}
These approximations are given by $[\A^* \A]_N=\left\{ [A_k^* A_l]_N: L^2_l \to L^2_k, k,l=0,1,...,p \right\}$
so that:
\begin{align} 
[A_0^{*} A_0]_N u&=u  \nonumber \\
[A_0^{*} A_l]_N u&=\frac1N \sum_{i=1}^N \int_{\I^l} u(s_1,...,s_l) \prod_{j=1}^l X_i(s_j) d\mu(s_j), \nonumber \\
[A_k^{*} A_l]_N u(s_1,...,s_k)&=\frac1N \sum_{i=1}^N \prod_{j=1}^k X_i(s_j) \int_{\I^l} u(\tilde{s}_1,...,\tilde{s}_l) \prod_{m=1}^l X_i(\tilde{s}_m) d\mu(\tilde{s_m}), \nonumber \\
k,l&=1,...,p. \label{eq:def_a*a_empirical}
\end{align}
 

and $[\A^* Y]_N=([A_0^* Y]_N,...,[A_p^* Y]_N) \in \CalL^2$, so that 
\begin{align}
[A_0^* Y]_N&=\frac1N \sum_{i=1}^N Y_i
\nonumber \\
[A_l^* Y]_N(s_1,...,s_l)&= \frac1N \sum_{i=1}^N Y_i \prod_{j=1}^l X_i(s_j), \nonumber \\
l&=1,...,p. \label{eq:def:a*y_emp}
\end{align}

In Section \ref{sec:tik_computations} we will use the above approximations for computing Tikhonov and interated Tikhonov regularizations in the context of PFR. 

We will also use the fact that for any $u \in \CalL^2$
\begin{align} \label{eq:norm_comparisons1}
\norm{\A u}_{L^2(\Omega, \mathbb{P})}=\norm{\sqrt{\A^* \A} u}_{ \CalL^2},
\end{align}
which follows immediately from the polar decomposition of the operator $\A$.




\subsection{Operator norms and related auxiliary estimates}
Let us move on by collecting several estimates related to the norms of the previously discussed operators. 
To this end let us also introduce the notion of effective dimension:
\begin{definition}
For a compact and self-adjoint operator $T$ on a Hilbert space $\Hs$ with countable sequence of nonnegative eigenvalues 
the associated \textit{effective dimension} is defined as
\begin{align*}
\N(T, \lambda)=\text{Tr}((T+\lambda \I)^{-1}T)=\sum_{j=1}^\infty \frac{\sigma_j}{\lambda+\sigma_j},
\end{align*}
see~\cite{caponnetto2007optimal}. We will use the abbreviation $\N(\lambda)=\N(\A^* \A, \lambda)$.
\end{definition}

The proof of the following lemma is a simple exercise.

\begin{lemma} \label{lem:hs_bounds}
Let $\text{HS}(\Hs_1, \Hs_2)$ denote the Hilbert space of Hilbert-Schmidt operators between Hilbert spaces $\Hs_1$ and $\Hs_2$. For simplicity let us also use $\text{HS}(\Hs_1, \Hs_1)=\text{HS}(\Hs_1).$
Under Assumption \ref{ass:unif} for $1 \le k,l \le p$ we have that
\begin{align*} 
 \norm{[A_l]_N}_{\text{HS}(L^2_l,\Real))}=\norm{[A_l^*]_N}_{\text{HS}(\Real,L^2_l)}  &\le \kappa^l,
 \\
\norm{A_l}_{\text{HS}(L^2_l,L^2(\Omega, \mathbb{P}))}  &\le \kappa^l, 
\\
\norm{A_k^* A_l}_{\text{HS}(L^2_l, L^2_k)} , \norm{[A_k^* A_l]_N}_{\text{HS}(L^2_l, L^2_k)} 
&\le \kappa^{l+k} 
\\
\norm{\A}_{\text{HS}(\CalL^2, L^2(\Omega, \mathbb{P}))} 
&\le  \tilde{\kappa}=:\sum_{l=0}^p \kappa^l 
\\
\norm{\A^* \A}_{\text{HS}(\CalL^2)}, \norm{[\A^* \A]_N}_{\text{HS}(\CalL^2)} 
&\le  \tilde{\kappa}^2 
\end{align*}
\end{lemma}
As a next step, we will collect several auxiliary estimates, which are derived by extending the arguments from \cite{guo2017learning} to our PFR-setting. To this end let us introduce the quantities
\begin{align*}
S(N, \lambda)=\frac{2  \tilde{\kappa} }{\sqrt{N}}\left(\frac{\tilde{\kappa} }{\sqrt{N \lambda}}+\sqrt{\N(\lambda)}\right),
\end{align*}
and
\begin{align*}
\Upsilon(\lambda)=\left(\frac{S(N, \lambda)}{\sqrt{\lambda}}\right)^2+1,
\end{align*}
which will appear in the subsequent estimates.
\begin{lemma} \label{lem:op_est_0}
For any $\delta \in (0,1)$, with confidence at least $1-\delta$ we have that 

\begin{align} \label{eq:op_est_0}
\left\| \A^* \A -[\A^* \A]_N  \right\|_{\CalL^2 \to \CalL^2} \le \left\| \A^* \A -[\A^* \A]_N  \right\|_{\text{HS}(\CalL^2)} \leq \frac{4 \tilde{\kappa}^2}{\sqrt{N}} \log \frac{2}{\delta}
 \end{align}
 \end{lemma}
 
\begin{lemma} \label{lem:op_est_1}
For any $\delta \in (0,1)$, with confidence at least $1-\delta$ we have that 

\begin{align} \label{eq:op_est_1}
\left\|(\lambda \I+\A^* \A )^{-1 / 2}\left(\A^* \A -[\A^* \A]_N \right)\right\|_{\CalL^2 \to \CalL^2} & \leq S(N,\lambda) \log \frac{2}{\delta}
 \end{align}
 \end{lemma}

 \begin{lemma} \label{lem:op_est_2_3}
 For any $\delta \in (0,1)$, with confidence at least $1-\delta$ we have that in case of noise assumption \eqref{eq:noise_1}:
\begin{align} \label{eq:op_est_2}
    \left\|(\lambda \I+\A^* \A)^{-1 / 2}([\A^* \A]_N u^+ -[\A^* Y]_N) \right\|_{\CalL^2} \leq  \frac{\sigma}{\tilde{\kappa} \delta} S(N, \lambda),
\end{align}
whereas in case of \eqref{eq:noise_2}:
\begin{align} \label{eq:op_est_3}
    \left\|(\lambda \I+\A^* \A)^{-1 / 2}([\A^* \A]_N u^+ -[\A^* Y]_N) \right\|_{\CalL^2} \leq  \frac{(M+\sigma) \log (2 / \delta)}{\tilde{\kappa} } S(N,\lambda).
\end{align}
\end{lemma}
We additionally need the following upper bounds (see e.g. \cite{guo2017learning}[Propositon 1]):
\begin{lemma} \label{lem:op_est_4}
For any $\delta \in (0,1)$, with confidence at least $1-\delta$ we have that 
\begin{align} \label{eq:op_est_4}
\left\|(\lambda \I+\A^* \A)\left(\lambda \I+[\A^* \A]_N\right)^{-1}\right\|_{\CalL^2 \to \CalL^2}\leq 2\left[\left(\frac{S(N, \lambda) \log \frac{2}{\delta}}{\sqrt{\lambda}}\right)^2+1\right]=\Xi
\end{align}
\end{lemma}
As a consequence of the operator concavity of the function $t \mapsto t^{1 / 2}$, we obtain, see e.g. \cite[Lemma A7]{blanchard2010optimal}, that
\begin{align} \label{eq:op_est_5}
\left\|(\lambda \I+\A^* \A)^{1 / 2}\left(\lambda \I+[\A^* \A]_N \right)^{-1 / 2}\right\|_{\CalL^2 \to \CalL^2} \leq \Xi^{1 / 2}.
\end{align}
To prove \eqref{eq:op_est_0}--\eqref{eq:op_est_4} we will also need the following concentration inequalities.
\begin{lemma}[\cite{pinelis1994optimum}]
\label{lem:concentration1}
Let $\Hs$ be a Hilbert space and $\xi$ be a random variable with values in $\Hs$. Assume that $\|\xi\|_{\Hs} \leq M$ almost surely. Let $\left\{\xi_1, \xi_2, \ldots, \xi_N\right\}$ be a sample of $N$ independent observations for $\xi$. Then for any $0<\delta<1$,
\begin{align*}
\left\|\frac{1}{N} \sum_{i=1}^N\left[\xi_i-\EXP(\xi)\right]\right\|_{\Hs} \leq \frac{2 M \log (2 / \delta)}{N}+\sqrt{\frac{2 \EXP \left(\|\xi\|_{\Hs}^2\right) \log (2 / \delta)}{N}}
\end{align*}
with confidence at least $1-\delta$.
\end{lemma}
\begin{lemma}[see e.g. Theorem 3.3.4. in \cite{yurinsky1995sums}] 
\label{lem:concentration}
 Let $\xi$ be a random variable with values in a Hilbert space $\Hs$. Let $\left\{\xi_1, \xi_2, \ldots, \xi_N\right\}$ be a sample of $N$ independent observations for $\xi$.
Furthermore, assume that the bound $\EXP\|\xi\|_{\Hs}^{\tilde{m}} \leqslant \frac{v}{2} \tilde{m} ! u^{\tilde{m}-2}$ holds for every $2 \leqslant \tilde{m} \in \mathbb{N}$, then for any $0<\delta<1$,
\begin{align*}
\left\|\frac{1}{N} \sum_{i=1}^N\left[\xi_i-\mathbb{E}(\xi)\right]\right\|_{\Hs} \leqslant \frac{2 u \log (2 / \delta)}{N}+\sqrt{\frac{2 v \log (2 / \delta)}{N}}
\end{align*}
with confidence at least $1-\delta$.
\end{lemma}
\begin{proof}[Proof of Lemma \ref{lem:op_est_0}]
 Consider the matrix of operators
\begin{align*} 
\mathcal{A}(\omega)&=\left\{ \mathcal{A}_{k,l}(\omega):  L^2_l \to L^2(\Omega, \mathbb{P}) \otimes L^2_k, k,l=0,...,p \right\},
\end{align*}
where $\mathcal{A}_{0,0}u(\omega) =u,\;\mathcal{A}_{0,l}u (\omega) = (A_l u)(\omega)$,
\begin{align}
\mathcal{A}_{k,l}u(\omega,s_1,...,s_k)
&=\prod_{j=1}^k X(\omega, s_j) \int_{\I^l} u(\tilde{s}_1,...,\tilde{s}_l) \prod_{m=1}^l X(\omega, \tilde{s}_m) d \tilde{s}_m  \label{eq:def:calA_entries}\\
k,l&=1,...,p, \omega \in \Omega \nonumber.
\end{align} 
Then in the spirit of Lemma \ref{lem:concentration1}, the operators $ \mathcal{A}^i$, $i=1,...,N$, defined by using  $X_i(\cdot)$ instead of $X(\omega, \cdot)$ in the above formulas, can be seen as independent observations of $\mathcal{A}(\omega)$.
 
It is clear that $\EXP(\mathcal{A}(\omega))=\A^* \A$ 
and that $\norm{\mathcal{A}(\omega)}_{\text{HS}(\CalL^2)} \le \tilde{\kappa}^2$, so that $\mathcal{A}(\omega)$ is a random variable in $\text{HS}(\CalL^2)$. It remains to apply Lemma \ref{lem:concentration1} in a straightforward way (setting $\xi(\omega)=\mathcal{A}(\omega)$, $\xi_i=\mathcal{A}^i$ and observing that $\EXP(\norm{\mathcal{A}(\omega)}_{\text{HS}(\CalL^2)}^2) \le \tilde{\kappa}^4$).

\end{proof}
\begin{proof}[Proof of Lemma \ref{lem:op_est_1}] 


 Consider the random variable $\xi(\omega)= (\A^* \A+ \lambda \mathbb{I} )^{-\frac12} \mathcal{A}(\omega)$ taking values in $\text{HS}(\CalL^2)$ (we will give a bound on the HS-norm below). Again in the spirit of Lemma \ref{lem:concentration1}, the operators $\xi_i= (\A^* \A+ \lambda \mathbb{I} )^{-\frac12} \mathcal{A}^i$, $i=1,...,N$
 can be seen as independent observations of $\xi(\omega)$. Also recall from Lemma \ref{lem:op_est_0} that $\EXP(\mathcal{A}(\omega))=\A^* \A$ and that $\norm{\mathcal{A}(\omega)}_{\text{HS}(\CalL^2)} \le \tilde{\kappa}^2$.

Then we observe $\EXP (\xi(\omega)) = (\lambda \I+\A^* \A )^{-1 / 2}\A^* \A$. Moreover it is evident from the spectral calculus for self-adjoint operators that
\begin{align*}
    \norm{\xi(\omega)}_{\text{HS}(\CalL^2)}&\le  \norm{(\A^* \A+ \lambda \mathbb{I} )^{-\frac12}}_{\CalL^2 \to \CalL^2} \norm{\mathcal{A}(\omega)}_{\text{HS}(\CalL^2)}  \le \frac{\tilde{\kappa}^2}{\sqrt{\lambda}}.
\end{align*}
Let now $\left\{ \varphi_m \right\}_m$ be an orthonormal basis in $\CalL^2$ that contains the eigenvectors of $\A^* \A$ corresponding to eigenvalues $\left\{ \sigma_m \right\}_m$. Then
\begin{align}
    \norm{\xi(\omega)}_{\text{HS}(\CalL^2)}^2&=\sum_m \norm{(\A^* \A+ \lambda \mathbb{I} )^{-\frac12} \mathcal{A}(\omega) \varphi_m}_{\CalL^2}^2 \nonumber \\
    &=\sum_{m,n}\innerpro{(\A^* \A+ \lambda \mathbb{I} )^{-\frac12} \mathcal{A}(\omega) \varphi_m,\varphi_n}_{\CalL^2}^2 \nonumber \\
    &=  \sum_{m,n} \innerpro{ \mathcal{A}(\omega) \varphi_m,(\A^* \A+ \lambda \mathbb{I} )^{-\frac12} \varphi_n}_{\CalL^2}^2 \nonumber \\
    &= \sum_{m,n} \frac{\innerpro{ \mathcal{A}(\omega) \varphi_m, \varphi_n}_{\CalL^2}^2}{\lambda+\sigma_n}.   \label{eq:xi_est}
\end{align}
Taking expectation in \eqref{eq:xi_est} and using the Cauchy-Schwarz inequality we obtain:
\begin{align*}
    \EXP (\norm{\xi(\omega)}_{\text{HS}(\CalL^2)}^2) 
    &\le  \EXP \left( \sum_{m,n}  \norm{\mathcal{A}(\omega)}_{\text{HS}(\CalL^2)} \norm{\varphi_m}_{\CalL^2} \norm{\varphi_n}_{\CalL^2}\frac{\innerpro{\mathcal{A}(\omega) \varphi_m,\varphi_n}_{\CalL^2}}{\lambda+\sigma_n} \right) \\
    &\le   \tilde{\kappa}^2 \EXP \left( \sum_{m,n}   \frac{\innerpro{\mathcal{A}(\omega) \varphi_m,\varphi_n}_{\CalL^2}}{\lambda+\sigma_n} \right)
    =\tilde{\kappa}^2  \sum_{m,n}   \frac{\innerpro{\EXP(\mathcal{A}(\omega)) \varphi_m,\varphi_n}_{\CalL^2}}{\lambda+\sigma_n}\\    &=\tilde{\kappa}^2 \sum_n \frac{\sigma_n}{\lambda+\sigma_n}=   \tilde{\kappa}^2 \mathcal{N}(\lambda).
\end{align*}
Now the application of  Lemma \ref{lem:concentration1} for $\xi(\omega)= (\A^* \A+ \lambda \mathbb{I} )^{-\frac12} \mathcal{A}(\omega)$ and $\xi_i= (\A^* \A+ \lambda \mathbb{I} )^{-\frac12} \mathcal{A}^i$ yields the desired bound.
\end{proof}

\begin{proof} [Proof of Lemma \ref{lem:op_est_2_3}]
Let us first focus on the more involved estimate \eqref{eq:op_est_3}. The estimate \eqref{eq:op_est_2} can be proven by similar reasoning. Consider $\mathcal{X}(\omega) \in L^2(\Omega,\mathbb{P} ) \otimes \CalL^2$,
\begin{align} \label{eq:def:calX}
\mathcal{X}(\omega)&= (\mathcal{X}_k(\omega))^p_{k=0}, \; \;\mathcal{X}_0(\omega)=1, \; \; \mathcal{X}_k(\omega)= \prod_{j=1}^k X(\omega,s_j), \; \; k=1,...,p,
\end{align}
with $\norm{\mathcal{X}(\omega)}_{\CalL^2} \le \tilde{\kappa}$, and the $\CalL^2$-valued random variable 

\begin{align*}
\xi(\omega)&=
(\lambda+\A^* \A)^{-\frac12}\mathcal{X}(\omega)(Y(\omega) - \A(\omega)u^+) 
=(\lambda+\A^* \A)^{-\frac12}\mathcal{X}(\omega) \varepsilon(\omega),
\end{align*}
where the last equality is due to Assumption \ref{as:noise_model}. 
Then in the spirit of Lemma \ref{lem:concentration}, the functions 
\begin{align*}
\xi_i= (\A^* \A+ \lambda \mathbb{I} )^{-\frac12}(\X^i Y_i -\mathcal{A}^i u^+),
\end{align*}
where $\mathcal{X}^i$ are defined by using  $X_i(\cdot)$ instead of $X(\omega, \cdot)$ in \eqref{eq:def:calX}, can be seen as independent observations of $\xi(\omega)$. Moreover we have:
\begin{align*}
\frac1N \sum_{i=1}^N \xi_i=(\A^* \A+ \lambda \mathbb{I} )^{-\frac12}\left( \frac1N \sum_{i=1}^N \X^i Y_i-\frac1N \sum_{i=1}^N \mathcal{A}^i u^+ \right),
\end{align*}
so that for $k=0,...,p$ recalling \eqref{eq:def:a*y_emp}:
\begin{align*}
\left(\frac1N \sum_{i=1}^N \X^i Y_i \right)_k(s_1,...,s_k)= \frac1N \sum_{i=1}^N  Y_i \prod_{j=1}^k X_i(s_j)=[A_k^* Y]_N(s_1,...,s_k)
\end{align*}
and recalling \eqref{eq:def_a*a_empirical}:
\begin{align*}
&\left(\frac1N \sum_{i=1}^N \mathcal{A}^i u^+ \right)_k (s_1,...,s_k) \\
&=\frac1N \sum_{i=1}^N  \sum_{l=0}^p \prod_{j=1}^k X_i(s_j) \int_{\I^l} u_l^+(\tilde{s}_1,...,\tilde{s}_l) \prod_{m=1}^l X_i(\tilde{s}_m) d\mu(\tilde{s}_m) \\
&=([\A^* \A]_N u^+)_k(s_1,...,s_k),
\end{align*}
which allows us to conclude:
\begin{align*}
\frac1N \sum_{i=1}^N \xi_i=(\lambda \I+\A^* \A)^{-1 / 2}([\A^* Y]_N-[\A^* \A]_N u^+).
\end{align*}
Due to Assumption \ref{as:noise_model} we have
\begin{align*}
\EXP(\xi(\omega))= \EXP((\lambda+\A^* \A)^{-\frac12}\mathcal{X}(\omega))\EXP(\varepsilon(\omega))=0,
\end{align*}
and moreover, due to \eqref{eq:CalApdef} for any $u \in \CalL^2$, $\omega \in \Omega$, 
\begin{align} \label{eq:subresult_calb}
\mathcal{X}(\omega) \innerpro{\mathcal{X}(\omega),u}_{\CalL^2}=\mathcal{X}(\omega) \sum_{l=0}^p \innerpro{\mathcal{X}_l(\omega),u_l}_{L^2_l}
=\mathcal{A}(\omega)u.
\end{align}
To provide more details on \eqref{eq:subresult_calb}, let us compute the entries of $\mathcal{A}(\omega)u \in \CalL^2$ explicitly, so that we have that for any $k=0,...,l$ by \eqref{eq:def:calA_entries}:
\begin{align*}
(\mathcal{A}u)_k(\omega, s_1,...,s_k)
&=\sum_{l=0}^p (\mathcal{A}_{kl} u_l)(\omega, s_1,...,s_k) \\
&=\sum_{l=0}^p \prod_{j=1}^k X(\omega, s_j) \int_{\I^l} u_l(\tilde{s}_1,...,\tilde{s}_l) \prod_{m=1}^l X(\omega, \tilde{s}_m) d \tilde{s}_m\\
&=\prod_{j=1}^k X(\omega, s_j) \sum_{l=0}^p \innerpro{\prod_{m=1}^l X(\omega, \tilde{s}_m), u_l}_{L^2_l}\\
&=\X_k(\omega, s_1,...,s_k) \sum_{l=0}^p \innerpro{\X_l(\omega), u_l}_{L^2_l}\\
&=\X_k(\omega, s_1,...,s_k)  \innerpro{\X(\omega), u}_{\CalL^2}.
\end{align*}
Next using the spectral calculus for self-adjoint operators:
\begin{align*}
    \norm{(\lambda+\A^* \A)^{-\frac12}\mathcal{X}(\omega)}_{\CalL^2} &\le  \norm{(\A^* \A+ \lambda \mathbb{I} )^{-\frac12}}_{\CalL^2 \to \CalL^2}   \norm{\mathcal{X}(\omega)}_{ \CalL^2}  
    \le \frac{\tilde{\kappa}}{\sqrt{\lambda}}.
\end{align*}

Now let $\left\{ \varphi_m \right\}_m$ be an orthonormal basis of $\CalL^2$ consisting of eigenvectors of $\A^* \A$ with corresponding eigenvalues $\left\{ \sigma_m \right\}_m$. Then we can compute the following:
\begin{align}
    &\norm{(\lambda+\A^* \A)^{-\frac12}\mathcal{X}(\omega)}_{\CalL^2}^2 \nonumber \\
    &= \sum_{m} \innerpro{(\A^* \A+ \lambda \mathbb{I} )^{-\frac12} \mathcal{X}(\omega),\varphi_m}_{\CalL^2}^2 \nonumber \\
    &=\sum_{m} \innerpro{\mathcal{X}(\omega),(\A^* \A+ \lambda \mathbb{I} )^{-\frac12} \varphi_m}_{\CalL^2}^2 = \sum_{m} \frac{\innerpro{\mathcal{X}(\omega), \varphi_m}_{\CalL^2}^2}{\sigma_m+\lambda}.\label{eq:subresult}
\end{align}

Taking expectation in \eqref{eq:subresult} and using \eqref{eq:subresult_calb} we can deduce:
\begin{align*}
&\EXP\left( \sum_{m} \frac{\innerpro{\mathcal{X}(\omega) \innerpro{\mathcal{X}(\omega) , \varphi_m}_{\CalL^2} , \varphi_m}_{\CalL^2}}{\sigma_m+\lambda} \right)=\EXP\left( \sum_{m} \frac{\innerpro{ \mathcal{A}(\omega) \varphi_m,  \varphi_m}_{\CalL^2}}{\sigma_m+\lambda} \right)\\&= \sum_{m} \frac{\innerpro{  \EXP( \mathcal{A}(\omega)) \varphi_m, \varphi_m}_{\CalL^2}}{\sigma_m+\lambda}
=\sum_{m} \frac{\norm{ \varphi_m}_{\CalL^2}^2 \sigma_m}{\sigma_m+\lambda} = \mathcal{N}(\lambda).
\end{align*}
By independence of $\varepsilon$ and $\mathcal{X}$ this allows to conclude for any $\omega \in \Omega$:
\begin{align*}
\EXP(\norm{\xi}^{\tilde{m}}_{\CalL^2}) &\le  \norm{(\lambda+\A^* \A)^{-\frac12}\mathcal{X}(\omega)}_{ \CalL^2}^{\tilde{m}-2} \\
&\cdot \EXP\left( \norm{(\lambda+\A^* \A)^{-\frac12}\mathcal{X}(\omega)}_{\CalL^2}^2\right) \EXP(|\varepsilon(\omega)|^{\tilde{m}}) \\
&\le \frac{\sigma^2 \mathcal{N}(\lambda)}{2} \left( \frac{M\tilde{\kappa}}{\sqrt{\lambda}}\right)^{\tilde{m}-2} \tilde{m}!.
\end{align*}
Now the application of  Lemma \ref{lem:concentration} for $\xi(\omega)$ and $\xi_i$ yields the desired bound \eqref{eq:op_est_3}. 

To obtain \eqref{eq:op_est_2} we need to follow the same lines of proof, but only consider the case $\tilde{m}=2$ and afterward apply Tschebyshev's inequality instead.
\end{proof}

\begin{proof}[Proof of Lemma \ref{lem:op_est_4}]
Using the relation 

\begin{align*}
    BA^{-1}=(B-A)B^{-1}(B-A)A^{-1}+(B-A)B^{-1}+I
\end{align*}
for the product of invertible operators $A$ and $B$ on Banach spaces, together with \eqref{eq:op_est_1} and the bounds $\norm{([\A^* \A]_N+ \lambda \I )^{-\frac12}}_{\CalL^2 \to \CalL^2} \le \frac{1}{\sqrt{\lambda}}$, $\norm{(\A^* \A+\lambda \I)^{-\frac12}}_{\CalL^2 \to \CalL^2} \le \frac{1}{\sqrt{\lambda}}$ (which follow from the spectral theorem), we get (using operator concavity of $t \mapsto \sqrt{t}$):
\begin{align*} 
    &\norm{\underbrace{(\A^* \A+\lambda \I)}_{B} \underbrace{([\A^* \A]_N+ \lambda \I )^{-1}}_{A^{-1}}}_{\CalL^2 \to \CalL^2}   \\
    &\le \norm{([\A^* \A]_N+ \lambda \I )^{-\frac12}}^2_{\CalL^2 \to \CalL^2} \cdot \norm{(\A^* \A+ \lambda \mathbb{I} )^{-\frac12}(\A^* \A-[\A^* \A]_N)}^2_{\CalL^2 \to \CalL^2} \\
    &+\norm{(\A^* \A+\lambda \I)^{-\frac12}}_{\CalL^2 \to \CalL^2} \cdot \norm{(\A^* \A+ \lambda \mathbb{I} )^{-\frac12}(\A^* \A-[\A^* \A]_N)}_{\CalL^2 \to \CalL^2} +1 \\
    &\le \frac{1}{\lambda}\norm{(\A^* \A+ \lambda \mathbb{I} )^{-\frac12}(\A^* \A-[\A^* \A]_N)}^2_{\CalL^2 \to \CalL^2}\\
    &+\frac{1}{\sqrt{\lambda}}\norm{(\A^* \A+ \lambda \mathbb{I} )^{-\frac12}(\A^* \A-[\A^* \A]_N)}_{\CalL^2 \to \CalL^2} +1 \nonumber \le 2\left[\left(\frac{S(N, \lambda) \log \frac{2}{\delta}}{\sqrt{\lambda}}\right)^2+1\right]
\end{align*}
with confidence at least $1-\delta$.
\end{proof}

\subsection{Statement and proof of main result}

In Subsection \ref{subsec:notation} we have already mentioned the necessity to employ a regularization in the context of PFR to find an approximation for the solution of \eqref{eq:noise_model} and started with considering Tikhonov regularization. In this section we want to formulate and prove our main result for even more general regularization scheme covering Tikhonov regularization as a particular case. That scheme is attributed to A. B. Bakushinskii (1967) and we introduce it here as follows (see, for example, \cite{lu2020balancing, mathe2003geometry}).
\begin{definition} \label{def:gen_reg}
A one-parametric family of functions $\{g_\lambda\}$ is called a regularization scheme for a bounded self-adjoint operator $B$ in a Hilbert space $\Hs$, if there are constants $\gamma_0, \gamma_{-1 / 2}$ and $\gamma_{-1}$ such that for any $\lambda$  meeting the inequalities $0<\lambda \leq \norm{B}_{\Hs \to \Hs}$ the following bounds hold: 
\begin{align}
\sup _{0<\sigma \le \norm{B}_{\Hs \to \Hs}} \sqrt{\sigma}\left|g_\lambda(\sigma)\right| \leq \frac{\gamma_{-1 / 2}}{\sqrt{\lambda}}  \label{eq:g_lam_sqrt_est} \\
\sup _{0<\sigma \le \norm{B}_{\Hs \to \Hs} }\left|g_\lambda(\sigma)\right| \leq \frac{\gamma_{-1}}{\lambda} \label{eq:g_lam_est} \\
\sup _{0<\sigma \le \norm{B}_{\Hs \to \Hs}}\left|1-\sigma g_\lambda(\sigma)\right| \leq \gamma_0 \label{eq:r_lam_est}
\end{align}

The \textit{qualification of the regularization scheme} $\left\{g_\lambda\right\}$ is the maximum value $q$ 
such that for some positive constant $\gamma_q$ it holds
\begin{align} \label{eq:qual_reg_scheme}
\sup _{0<\sigma \le \norm{B}_{\Hs \to \Hs}}\left|r_\lambda(\sigma) \sigma^q\right| \leq \gamma_q \lambda^q,
\end{align}
where
\begin{align*}
r_\lambda(\sigma)=1-\sigma g_\lambda(\sigma)
\end{align*}
is the so-called residual function. 
\end{definition}
Let us remark that one can also immediately deduce 
\begin{align} \label{eq:sigma_g_lam_est}
\sup _{0<\sigma \le \norm{B}_{\Hs \to \Hs}}\left|\sigma g_\lambda(\sigma)\right| \leq \gamma_0+1
\end{align}
from \eqref{def:gen_reg}.
Note that Tikhonov regularization \eqref{eq:tikpftfr}, \eqref{eq:tikpftfremp} fulfils Definition \ref{def:gen_reg} with $g_\lambda(\sigma)=(\sigma-\lambda)^{-1}$, 
and $q=1$, 
while iterated Tikhonov regularization with $\tilde{q}$ iterations corresponds to $g_\lambda(\sigma) = g_{\lambda,\tilde{q}}(\sigma)=\frac{1}{\sigma}(1-(\frac{\lambda}{\lambda+\sigma})^{\tilde{q}})$ and $q=\tilde{q}$.

Recall that under Assumption \ref{eq:ass:projection} the minimizer $u = u^+=(u_0^+,...,u_p^+)$ of \eqref{eq:pftfr} solves $\A^* \A u=\A^* Y$. Then in view of Definition \ref{def:gen_reg} a regularized approximate solution of the later equation can be written in the form
\begin{align*} 
    u^N_{\lambda}=g_{\lambda}([\A^* \A]_N)[\A^* Y]_N
\end{align*}
In order to account for the smoothness of the target vector function $u^+$, we will use the concept of general source conditions indexed by operator monotone functions in the context of learning (see, e.g., \cite{mathe2006regularization,bauer2007regularization} or \cite{lu2013regularization}[Section 2.8.2.] for an extended discussion), and thus need to introduce several notions:
\begin{definition}
A continuous non-decreasing function $\varphi: [0,a] \to \Real^+$ is called \textit{index function} if $\varphi(0)=0$.
An index function is called operator monotone if for any bounded non-negative self-adjoint operators $A,B$ on a Hilbert space $\mathcal{H}$ with 
\newline
$\norm{A}_{{\Hs \to \Hs}}, \norm{B}_{{\Hs \to \Hs}} \le a$, we have that $A \prec B$ implies $\varphi(A) \prec \varphi(B)$; here the symbol $\prec$ accounts for the corresponding inequalities of the associated quadratic forms, i.e., the relation $A \prec B$ means, for example, that for any $x$ from $\mathcal{H}$ it holds $\innerpro{Ax,x}_{\mathcal{H}}\le\innerpro{Bx,x}_{\mathcal{H}}$.
\end{definition}
\begin{definition}
We say that an index function $\varphi: [0,a] \to \Real^+$ is covered by the qualification $q$ if the function $\sigma \mapsto \sigma^{q}/\varphi(\sigma)$ is non-decreasing on $[0,a]$.
\end{definition}
The following well known result tells us that operator monotone index functions are covered by qualification $1$:
\begin{lemma}[see e.g. Lemma 3 in \cite{mathe2006regularization} or Lemma 3.4. \cite{lu2020balancing}] \label{lem:op_mon}
For each operator monotone index function $\varphi$ there is a constant $1 \le C < \infty$ such that, whenever $0<s\le t \le \tilde{\kappa}^2$, we have
$\frac{\varphi(t)}{t} \le C \frac{\varphi(s)}{s}$.
\end{lemma}
Now we can state the smoothness assumption as follows:
\begin{assumption} \label{ass:smoothness}
For some operator monotone index function $\varphi: [0,\tilde{\kappa}^2] \to \Real^+$ we have that
\begin{align*} 
    u^+=\varphi(\A^* \A) v^+
\end{align*}
for some $v^+ \in  \CalL^2$.

\begin{remark}
Note that when $p=1$ (linear functional regression), the operator $\A^* \A$ is connected with the covariance of $X$ with itself, while for $p>1$, $\A^* \A$ is an example of the so-called high-order statistics, which use the third or higher power of $X$. There is an extensive literature on such statistics, see e.g. \cite{swami1997bibliography}.
\end{remark}

\end{assumption}


Our main goal is to estimate the excess risk, which 
in view of $\EXP(\varepsilon)=0$ and \eqref{eq:noise_model}, \eqref{eq:norm_comparisons1}, can be expressed in the $\CalL^2$-norm as follows:
\begin{align*} 
\mathcal{E}(u^+)-\mathcal{E}(u^N_{\lambda})=\norm{\A (u^+ - u^N_{\lambda})}_{L^2(\Omega, \mathbb{P}) }^2=\norm{\sqrt{\A^* \A} (u^+ - u^N_{\lambda})}_{\CalL^2}^2.
\end{align*}
Let us furthermore introduce the auxiliary approximant
\begin{align*} 
\bar{u}^N_{\lambda}=g_{\lambda}([\A^* \A]_N)[\A^* \A]_N u^+,
\end{align*}
such that
\begin{align*} 
  \norm{u^+ - u^N_{\lambda}}_{ \CalL^2}  \le& \underbrace{\norm{u^+-\bar{u}^N_{\lambda}}_{\CalL^2}}_{\text{Approximation Error}}+\underbrace{\norm{\bar{u}^N_{\lambda}- u^N_{\lambda}}_{\CalL^2}}_{\text{Propagation Error}},
\end{align*}
and similar for the excess risk.

The analysis below is performed using arguments similar to the ones employed in \cite{lu2020balancing}, but since the context here is different from \cite{lu2020balancing}, we present the main steps for the reader's convenience.
Let us start with a bound on the approximation error:
\begin{lemma} \label{lem:approx}
   Under the Assumptions \ref{ass:unif}, \ref{eq:ass:projection} and \ref{ass:smoothness} for a sufficiently high qualification of the regularization scheme $q \geq \frac32$ with probability $1-\delta$ we have the following bounds for the approximation error:

   \begin{align} \label{eq:bound_approx_hs_1}
  \norm{u^+ - \bar{u}^N_{\lambda}}_{ \CalL^2} \le  C(\gamma_0+\gamma_1)  \Xi \norm{v^+}_{ \CalL^2}  \varphi(\lambda),
\end{align}
\begin{align} \label{eq:bound_approx_excess2}
  \norm{\A(u^+ - \bar{u}^N_{\lambda})}_{L^2(\Omega, \mathbb{P}) }\le  C \sqrt{8}(\gamma_0+\gamma_{\frac32})  \Xi^{\frac32} \norm{v^+}_{ \CalL^2} \sqrt{\lambda} \varphi(\lambda),
\end{align}
for $0\le \lambda \le \tilde{\kappa}^2$ and $C$ from Lemma \ref{lem:op_mon}.
\end{lemma}
\begin{proof}
\begin{align*}
 &\norm{u^+ - \bar{u}^N_{\lambda}}_{\CalL^2} \le \norm{(I-g_\lambda(\A^* \A]_N)[\A^* \A]_N) \varphi(\A^* \A) }_{\CalL^2 \to \CalL^2} \norm{v^+}_{\CalL^2}\\  
 &\le  \norm{([\A^* \A]_N+ \lambda \I )r_{\lambda}([\A^* \A]_N)}_{\CalL^2 \to \CalL^2}
 \norm{([\A^* \A]_N+ \lambda \I )^{-1}(\A^* \A+ \lambda \mathbb{I} )}_{\CalL^2 \to \CalL^2}
\\&\cdot  \norm{(\A^* \A+ \lambda \mathbb{I} )^{-1} \varphi(\A^* \A)}_{\CalL^2 \to \CalL^2} \norm{v^+}_{\CalL^2}\\
 &\le  \Xi 
 \norm{v^+}_{\CalL^2} \underbrace{\norm{([\A^* \A]_N+ \lambda \mathbb{I} )r_{\lambda}([\A^* \A]_N)}_{\CalL^2 \to \CalL^2}}_{(I)}  \underbrace{\norm{(\A^* \A+ \lambda \mathbb{I} )^{-1} \varphi(\A^* \A)}_{\CalL^2 \to \CalL^2}}_{(II)}.
\end{align*}   
    
    The part (I) can now be estimated directly by the spectral calculus, \eqref{eq:r_lam_est}--\eqref{eq:qual_reg_scheme} and the triangle inequality as follows (see Lemma 7.1. in \cite{lu2020balancing}):
    \begin{align*}
        (I) \le (\gamma_0+\gamma_1) \lambda.
    \end{align*}
    For the part (II), we can use the argument from \cite{pereverzyev2022introduction}[Proposition 3.1.] or \cite{lu2020balancing} to infer
\begin{align} \label{eq:bound_on_II}
(II) \le C \frac{\varphi(\lambda)}{\lambda}.
\end{align}
For the reader's convenience we repeat that argument here. First notice that in view of Lemma \ref{lem:op_mon} we have
\begin{align*}
\sup_{\lambda \le \sigma \le \tilde{\kappa}^2} (\sigma+\lambda)^{-1} \varphi(\sigma)&\le\sup_{\lambda \le \sigma \le \tilde{\kappa}^2} 
(\sigma+\lambda)^{-1}\sigma  \sup_{\lambda \le \sigma \le \tilde{\kappa}^2} \frac{\varphi(\sigma)} {\sigma} 
\le C \frac{\varphi(\lambda)}{\lambda},
\end{align*}
which, together with obvious bound 
\begin{align*}
\sup_{0 \le \sigma \le \lambda} (\sigma+\lambda)^{-1} \varphi(\sigma)&\le\frac{1}{\lambda}
\sup_{0 \le \sigma \le \lambda} \varphi(\sigma) 
\le \frac{\varphi(\lambda)}{\lambda},
\end{align*}
gives us \eqref{eq:bound_on_II} and finally, by the spectral calculus, the bound \eqref{eq:bound_approx_hs_1}.


For \eqref{eq:bound_approx_excess2} we argue as follows:
\begin{align*}
    &\norm{\A(u^+ - \bar{u}^N_{\lambda})}_{L^2(\Omega, \mathbb{P})}= \norm{(\A^* \A)^{\frac12}(g_\lambda([\A^* \A]_N)[\A^* \A]_N - \I)\varphi(\A^* \A) v^+)}_{\CalL^2}\\
    &\le \norm{(\A^* \A+ \lambda \mathbb{I} )^{\frac12}([\A^* \A]_N+ \lambda \I )^{-\frac12}}_{\CalL^2 \to \CalL^2} \underbrace{\norm{([\A^* \A]_N+ \lambda \I )^{\frac32} r_{\lambda}([\A^* \A]_N)}_{\CalL^2 \to \CalL^2}}_{(III)} \\
    &\cdot \norm{([\A^* \A]_N+ \lambda \I )^{-1}(\A^* \A+ \lambda \mathbb{I} )}_{\CalL^2 \to \CalL^2} \norm{(\A^* \A+ \lambda \mathbb{I} )^{-1} \varphi(\A^* \A)}_{\CalL^2 \to \CalL^2} \norm{v^+}_{\CalL^2}.
    \end{align*}
Recall that the qualification $q$ of the employed regularization is assumed to be high enough, such that $q \geq \frac32$. Then Hölder’s inequality (i.e. in our case the inequality $(a+b)^c \le 2^c(a^c+b^c)$ for $a,b \geq 0$, $c>0$) together with the spectral calculus and \eqref{eq:r_lam_est}--\eqref{eq:qual_reg_scheme} yields
    \begin{align*}
    (III) \le \sup_{0 \le \sigma \le \tilde{\kappa}^2}   (\sigma+\lambda)^{\frac32} r_{\lambda}(\sigma) \le \sup_{0 \le \sigma \le \tilde{\kappa}^2} 2^{\frac32}  (\sigma^{\frac32}+\lambda^{\frac32}) r_{\lambda}(\sigma) \le  \sqrt{8}(\gamma_0+\gamma_{\frac32}) \lambda^{\frac32}
    \end{align*}
Combining this bound with the estimate on (II) and \eqref{eq:op_est_4}, \eqref{eq:op_est_5}, we arrive at \eqref{eq:bound_approx_excess2}.
\end{proof}

Now we come to Propagation Error part.
\begin{lemma} \label{lem:prop}
   Under assumptions \ref{ass:unif}--\ref{eq:ass:projection} and \eqref{eq:noise_1}, with probability $1-\delta$ and $0\le \lambda \le \tilde{\kappa}^2$ we have the following bound for the propagation error:

\begin{align} 
   &\norm{u^N_{\lambda}-\bar{u}^N_{\lambda}}_{\CalL^2} \le \Xi^{\frac12} (\gamma_{-1 / 2} +\gamma_{-1}) \frac{\sigma}{\tilde{\kappa} \delta \sqrt{\lambda}} S(N, \lambda) \nonumber \\
 &\norm{\A(u^N_{\lambda}-\bar{u}^N_{\lambda})}_{L^2(\Omega, \mathbb{P})} \le \Xi (\gamma_0+\gamma_{-1}+1) \frac{\sigma}{\tilde{\kappa} \delta} S(N, \lambda),
 \label{eq:bound_prop_1}
\end{align}
whereas under \eqref{eq:noise_2}:
\begin{align} 
&\norm{u^N_{\lambda}-\bar{u}^N_{\lambda}}_{\CalL^2} \le \Xi^{\frac12} (\gamma_{-1 / 2} +\gamma_{-1})\frac{(M+\sigma) \log (2 / \delta)}{\tilde{\kappa} \sqrt{\lambda}} S(N, \lambda) \nonumber \\
 &\norm{\A(u^N_{\lambda}-\bar{u}^N_{\lambda})}_{L^2(\Omega, \mathbb{P})} \le \Xi (\gamma_0+\gamma_{-1}+1) \frac{(M+\sigma) \log (2 / \delta)}{\tilde{\kappa}} S(N,\lambda). \label{eq:bound_prop_2}
\end{align}
\end{lemma}

\begin{proof}
We use the following decomposition: 

\begin{align*}
&\norm{u^N_{\lambda}-\bar{u}^N_{\lambda}}_{\CalL^2}=\norm{g_\lambda([\A^* \A]_N)([\A^* \A]_N u^+ -[A^* Y]_N)}_{\CalL^2}
\\&\le \norm{(\A^* \A+ \lambda \mathbb{I} )^{\frac12}([\A^* \A]_N+ \lambda \I )^{-\frac12}}_{\CalL^2 \to \CalL^2} \underbrace{\norm{([\A^* \A]_N+ \lambda \I )^{\frac12} g_{\lambda}([\A^* \A]_N)}_{\CalL^2 \to \CalL^2}}_{(IV)}
\\
&\cdot   \norm{(\A^* \A+ \lambda \mathbb{I} )^{-\frac12}([\A^* \A]_N u^+ -[A^* Y]_N)}_{\CalL^2}
\end{align*}
and
\begin{align*}
&\norm{\A(u^N_{\lambda}-\bar{u}^N_{\lambda})}_{L^2(\Omega, \mathbb{P})}=\norm{(\A^* \A)^{\frac12}g_\lambda([\A^* \A]_N)([\A^* \A]_N u^+ -[A^* Y]_N)}_{\CalL^2} \\
&\le \norm{(\A^* \A+ \lambda \mathbb{I} )^{\frac12}([\A^* \A]_N+ \lambda \I )^{-\frac12}}_{\CalL^2 \to \CalL^2} \norm{([\A^* \A]_N+ \lambda \I )g_{\lambda}([\A^* \A]_N)}_{\CalL^2 \to \CalL^2} \\
&\cdot  \norm{([\A^* \A]_N+ \lambda \I )^{-\frac12} (\A^* \A+ \lambda \mathbb{I} )^{\frac12}}_{\CalL^2 \to \CalL^2} \norm{(\A^* \A+ \lambda \mathbb{I} )^{-\frac12}([\A^* \A]_N u^+ -[A^* Y]_N)}_{\CalL^2},
\end{align*}
so that we can deduce \eqref{eq:bound_prop_1} and \eqref{eq:bound_prop_2} by the spectral calculus, \eqref{eq:g_lam_est}--\eqref{eq:sigma_g_lam_est}, \eqref{eq:op_est_5} and \eqref{eq:op_est_2} or \eqref{eq:op_est_3}, respectively.
To deal with (IV) efficiently, similar to (III), we use the spectral calculus, H\"older's inequality,  and \eqref{eq:g_lam_sqrt_est}--\eqref{eq:g_lam_est} to obtain:
 \begin{align*}
 (IV) \le \sup_{0 \le \sigma \le \tilde{\kappa}^2}   (\sigma+\lambda)^{\frac12} g_{\lambda}(\sigma) \le \sup_{0 \le \sigma \le \tilde{\kappa}^2}   (\sigma^{\frac12}+\lambda^{\frac12}) g_{\lambda}(\sigma) \le  \frac{1}{\sqrt{\lambda}}(\gamma_{-1 / 2} +\gamma_{-1}).
 \end{align*}
\end{proof}

Combining Lemmas \ref{lem:approx}--\ref{lem:prop} and using some straightforward estimates finally allows us to deduce the following finite sample error bounds:
\begin{lemma} \label{lem:gen_bound_1}
  If the employed regularization scheme has qualification $q \geq \frac32$, then under assumptions \ref{ass:unif}, \ref{eq:ass:projection}, \ref{ass:smoothness} and noise assumption \eqref{eq:noise_1} with probability $1-\delta$ we have
  \begin{align} 
  \norm{u^+ - u^N_{\lambda}}_{\CalL^2} &\le  C \Upsilon(\lambda)^{\frac12} \left(\Upsilon(\lambda)^{\frac12} \varphi(\lambda) + S(N,\lambda) \frac{1}{\sqrt{\lambda}} \right) \frac{\log(\frac{4}{\delta})}{\delta}, \nonumber \\
     \mathcal{E}(u^N_{\lambda})-\mathcal{E}(u^+) &\le C \Upsilon(\lambda)^2 \left(\lambda \varphi(\lambda)^2 + S(N,\lambda)^2 \right) \frac{\log^4(\frac{4}{\delta})}{\delta^2},
     \label{eq:gen_bound_noise_1}
  \end{align}
  whereas under noise assumption \eqref{eq:noise_2} it holds that
    \begin{align} 
   \norm{u^+ - u^N_{\lambda}}_{\CalL^2} &\le  C \Upsilon(\lambda)^{\frac12} \left(\Upsilon(\lambda)^{\frac12} \varphi(\lambda) + S(N,\lambda) \frac{1}{\sqrt{\lambda}} \right)  \log^2(\frac{4}{\delta}), \nonumber \\
     \mathcal{E}(u^N_{\lambda})-\mathcal{E}(u^+) &\le C \Upsilon(\lambda)^2 \left(\lambda \varphi(\lambda)^2 + S(N,\lambda)^2 \right) \log^6(\frac{4}{\delta});
     \label{eq:gen_bound_noise_2}
  \end{align}
here $C$ is a generic constant independent of $N$, $\delta$ and $\lambda$.
\end{lemma}
Below we also employ the following statement:
\begin{lemma}
\label{lem:estimates}
    There exists a $\lambda^*$ satisfying $\frac{\mathcal{N}(\lambda^*)}{\lambda^*}=N$, such that for $\lambda^* \le \lambda \le \norm{\A^* \A}_{\CalL^2 \to \CalL^2} $ it holds:
    \begin{align}\label{eq:S_n_est}
        S(N,\lambda) \le \frac{2  \tilde{\kappa} }{\sqrt{N}}(\sqrt{2}  \tilde{\kappa}  +\sqrt{\N(\lambda)}),
    \end{align}
    and
    \begin{align}\label{eq:upsilon_est}
        \Upsilon(\lambda) \le 1+\left(4 \tilde{\kappa}^2+2 \tilde{\kappa}   \right)^2.
    \end{align}
\end{lemma}
For the sake of completeness the proof is also placed here.
\begin{proof}
It is easy to see 
that $\lambda \mapsto \frac{\N(\lambda)}{\lambda}$ is decreasing from $\infty$ to zero,
thus $\lambda^*$ exists and is well defined. Since $\N(\lambda)$ is also a decreasing function, we have $N=\frac{\N(\lambda^*)}{\lambda^*} \geq \frac{\N(\norm{\A^* \A}_{\CalL^2 \to \CalL^2})}{\lambda^*} \geq \frac{1}{2 \lambda^*}$, so that $N\lambda^* \geq \frac12$. This immediately implies \eqref{eq:S_n_est}. To obtain \eqref{eq:upsilon_est}, we observe the following:
\begin{align*}
    \frac{S(N,\lambda)}{\sqrt{\lambda}} \le \frac{S(N,\lambda^*)}{\sqrt{\lambda^*}} = 2 \tilde{\kappa} \left( \frac{\tilde{\kappa}}{N \lambda^*}+\sqrt{\frac{\N(\lambda^*)}{N \lambda^*}}\right) \le 2 \tilde{\kappa}(2\tilde{\kappa}+1).
\end{align*}
\end{proof}


A straightforward application of the previous lemma allows us to simplify the statement of Lemma \ref{lem:gen_bound_1}:
\begin{theorem} \label{thm:main}
Under the conditions of Lemma \ref{lem:gen_bound_1} it holds
  \begin{align*} 
  \norm{u^+ - u^N_{\lambda}}_{\CalL^2} &\le C P(\delta) \left(\varphi(\lambda) + \sqrt{\frac{\N(\lambda)}{N\lambda}} \right), 
  \\
     \mathcal{E}(u^N_{\lambda})-\mathcal{E}(u^+) &\le C Q(\delta) \left(\lambda \varphi(\lambda)^2 + \frac{\N(\lambda)}{N} \right), 
  \end{align*}
  where $C$ denotes a generic constant independent of $N$ and $\lambda$, while the forms of the coefficients $P(\delta), Q(\delta)$ depend on the noise assumptions and coincides with the corresponding $\delta$-depending quantities in \eqref{eq:gen_bound_noise_1}, \eqref{eq:gen_bound_noise_2}.
\end{theorem}
This immediately gives us the following corollary:
\begin{corollary} \label{cor:gen_bound_power_rates}
    Let the assumptions of Lemma \ref{lem:gen_bound_1} be satisfied. Suppose that $\varphi(\lambda)=\lambda^{r}$ for some $0 \le r \le 1$ 
    and that the effective dimension obeys $\N(\lambda) \le C_0 \lambda^{-\theta}$ for any $\lambda>0$. If we choose $\lambda=N^{\frac{-1}{2r+\theta+1}}$, then with probability at least $1-\delta$ under the noise assumption \eqref{eq:noise_1} we have
    \begin{align} 
    \norm{u^+ - u^N_{\lambda}}_{\CalL^2} &\le C N^{\frac{-r}{2r+\theta+1}}\frac{\log(\frac{4}{\delta})}{\delta}, \nonumber \\
    \label{eq:gen_bound_power_noise_1}
        \mathcal{E}(u^N_{\lambda})-\mathcal{E}(u^+) &\le C  N^{\frac{-2r-1}{2r+\theta+1}}\frac{\log^4(\frac{4}{\delta})}{\delta^2},
    \end{align}
    whereas under noise assumption \eqref{eq:noise_2} it holds that:
    \begin{align} 
    \norm{u^+ - u^N_{\lambda}}_{\CalL^2} &\le C N^{\frac{-r}{2r+\theta+1}}\log^2(\frac{4}{\delta}), \nonumber \\
    \label{eq:gen_bound_power_noise_2}
     \mathcal{E}(u^N_{\lambda})-\mathcal{E}(u^+) &\le C  N^{\frac{-2r-1}{2r+\theta+1}}\log^6(\frac{4}{\delta}).
  \end{align}
\end{corollary}
Some remarks are in order:
\begin{remark}
For the sake of simplicity in the above analysis we have discussed only the case of operator monotone index functions. At the same time, by just reusing the arguments from \cite{lu2020balancing, pereverzyev2022introduction} one can extend the obtained results to the case of index functions allowing a splitting into a product of operator monotone and Lipschitz functions. However, we decided not to include the details here, in order to not overload the notations.

\end{remark}
\begin{remark}
The corollary above generalizes and to some extend improves the result of \cite{tong2018analysis}, where the authors have considered the case of linear functional regression, i.e. $p=1$, regularized by the Tikhonov method \eqref{eq:tikpftfr}, \eqref{eq:tikpftfremp} under the assumption that $\N(\lambda) =\mathcal{O} (\lambda^{-\theta})$. The main result of \cite{tong2018analysis} states that $\mathcal{E}(u^N_{\lambda})-\mathcal{E}(u^+) = \mathcal{O}(  N^{\frac{-1}{1+\theta}})$, that is of lower order than \eqref{eq:gen_bound_power_noise_1}, \eqref{eq:gen_bound_power_noise_2}, because the smoothness of 
the target of approximation $u^+$ has not been taken into account in \cite{tong2018analysis}, and because we have assumed that the employed regularization scheme has a higher qualification $q>\frac32$ than the one for Tikhonov method, where $q=1$. On the other hand, by examining the above proofs for $q=1$, we may state that for $u_N^{\lambda}$ obtained by the Tikhonov method with $\lambda=N^{-\frac{1}{2r+\theta}}$, the excess risk $\mathcal{E}(u^N_{\lambda})-\mathcal{E}(u^+)$ is of order $\mathcal{O}(  N^{\frac{-2r}{2r+\theta}})$, which improves the result of \cite{tong2018analysis}, for sufficiently smooth $u^+$ meeting Assumption \ref{ass:smoothness} with $\varphi(\lambda)=\lambda^r$ for $\frac12 \textless r \le 1$. At the same time, as it can be seen in the next section, an iterated version of the Tikhonov method is computationally not much expensive as the standard scheme \eqref{eq:tikpftfr}--\eqref{eq:tikpftfremp}, but having the qualification $q\geq \frac32$, this method then leads to the improved bounds \eqref{eq:gen_bound_power_noise_1}--\eqref{eq:gen_bound_power_noise_2}. 
\end{remark}
\begin{remark}
Regularized linear functional regression (i.e., $p = 1$), but in the (more general) function-to-function setting, has been studied recently by \cite{mollenhauer2022learning} under the assumptions that $\varphi(\lambda)=\lambda^{r}$ and $\theta=1$. Extending that results to the polynomial function-to-function regression is an interesting avenue for future work.
\end{remark}


\section{Computations for Tikhonov regularization and its iterations}
\label{sec:tik_computations}
In this chapter, we, at first, describe how to find the regularized approximation $u_{\lambda}^N$ \eqref{eq:tikpftfremp} 
constructed in the form $u_{\lambda}^N=(u_{\lambda,0}^N,u_{\lambda,1}^N,...,u_{\lambda,p}^N) \in \CalL^2$, so that
\begin{align*}
u_{\lambda,0}^N&=b_0 \in \Real\\
u_{\lambda,l}^N(s_1,...,s_l)&=\sum_{i=1}^N b_{l,i} \prod_{j=1}^l X_i(s_j) \in L^2_l, \quad l=1,...,p.
\end{align*}
Inserting this form into \eqref{eq:tikpftfremp} and equating the corresponding coefficients now result in the following system of linear equations for $b_0$ and for $b_{k,i}$ with $k=1,...,p$ and $i=1,...,N$:
\begin{align} 
(\lambda+1)b_0+\frac1N \sum_{i=1}^N \sum_{l=1}^p \sum_{s=1}^N b_{l,s} c_{i,s}^l &= \frac1N \sum_{i=1}^N Y_i,\label{eq:b_0_coeff} 
\end{align}
\begin{align} 
&\lambda b_{k,i}+\frac1N b_0+\frac1N \sum_{l=1}^p \sum_{s=1}^N b_{l,s} c_{i,s}^l = \frac1N  Y_i, \label{eq:b_ki_coeff}
\end{align}
where $c_{i,s}=\int_{\I} X_i(\tilde{s}) X_s(\tilde{s}) d\mu(\tilde{s})$. 

Next we observe that the form of the equations for the coefficients $b_0, b_{i,k}$ allows for a reduction to a system of only $(N+1)$ equations. Indeed, if we sum up the equations \eqref{eq:b_ki_coeff} for all $i=1,...N$ and then subtract the resulting sum from \eqref{eq:b_0_coeff}, 
we obtain 
\begin{align*}
b_0= \sum_{i=1}^N b_{k,i},
\end{align*}
Moreover, \eqref{eq:b_ki_coeff} allows for the conclusion that $ b_{k,i}=b_{1,i}$ for any $k=1,...,p$, because \eqref{eq:b_ki_coeff} can be rewritten as
\begin{align*}
    \lambda b_{k,i}=\frac1N  Y_i-\frac1N b_0-\frac1N \sum_{l=1}^p \sum_{s=1}^N b_{l,s} c_{i,s}^l,
\end{align*}
where the right-hand side does not depend on $k$.

Thus, the coefficients of polynomial functional regression regularized by the Tikhonov method with a single regularization parameter are completely defined by its first (linear) term.


Recall that the Tikhonov regularization scheme has the qualification $q=1$, while Lemma \ref{lem:gen_bound_1} and Theorem  \ref{thm:main} suggest to use of a regularization  with the qualification $q \geq \frac32$.
Therefore, we consider the iterated Tikhonov regularization with $q \geq 2$ being the number of iteration steps:
\begin{align*} 
\lambda u_{\lambda,q}^N + [\A^* \A]_N u_{\lambda,q}^N &= \lambda u_{\lambda,q-1}^N +[\A^* Y]_N , \\
u_{\lambda,0}^N &=0,
\end{align*}
where $u_{\lambda,q}^N=(u_{\lambda,0,q}^N,u_{\lambda,1,q}^N,...,u_{\lambda,p,q}^N) \in \CalL^2$, and 
\begin{align*}
u_{\lambda,0,q}^N&=b_{0,q} \in \Real,\\
u_{\lambda,l,q}^N(t,s_1,...,s_l)&=\sum_{i=1}^N b_{l,i,q} \prod_{j=1}^l X_i(s_j) \in L^2_l, \quad l=1,...,p.
\end{align*}
To construct $u_{\lambda,q}^N$ one needs to solve the following system of linear equations for coefficients $b_{0,q}, b_{l,i,q} $ :

\begin{align*} 
(\lambda+1)b_{0,q}+\frac1N \sum_{i=1}^N \sum_{l=1}^p \sum_{s=1}^N b_{l,s,q} c_{i,s}^l &=  \lambda  b_{0,q-1}+\frac1N \sum_{i=1}^N Y_i, \\
\lambda b_{k,i,q}+\frac1N b_{0,q}+\frac1N \sum_{l=1}^p \sum_{s=1}^N b_{l,s,q} c_{i,s}^l &= \lambda b_{k,i,q-1} +\frac1N  Y_i, \\ k=1,...,p, \quad i=1,...,N.
\end{align*}
Here, by the same reasoning as above, we also see that this reduces to a system of $(N+1)$ equations, from which only the coefficients $b_{0,q},b_{1,i,q}$ of the components,  $u_{\lambda,0,q}^N,u_{\lambda,1,q}^N$ need to be found.

Overall, we can observe that a single-parameter regularization in the context of PFR puts emphasis on the linear part of that regression. This can be seen as a limitation. One possibility to overcome it is to use a multiple-parameter regularization scheme, which we plan to discuss in a forthcoming work. But even a multiple-parameter regularization will still reduce functional regression to a linear problem. In order to go beyond the linearity, one can employ deep neural networks, which have been discussed recently in \cite{yao2021deep, shi2024nonlinear}. Thus there are many ways to explore in functional data analysis.

\section{Experimental evaluation}
In this section we demonstrate the action of the above presented algorithm on a toy example.
To this end, as an explanatory variable, we consider a random process
\begin{align*}
X(\omega, t)=\sum_{k=0}^5 \xi_k(\omega) \cos (k t), t \in[0,2 \pi],
\end{align*}
where $\xi_k(\omega)$ are random variables uniformly distributed on $[-1,1]$.
Consider also the response variable $Y(\omega)$ related to the explanatory variable $X(\omega, t)$  as follows:

\begin{align}
 Y(\omega)&=u_0^+ +\int_0^{2 \pi} X(\omega, t) u_1^+(t) d \mu(t) \nonumber\\
&+\int_0^{2 \pi} \int_0^{2 \pi} X(\omega, t) X(\omega, \tau) u_2^+(t, \tau) d \mu(t) d\mu( \tau) . \label{eq:y_def_exp}
\end{align}
In our simulations, we use $u^{+}=\left(u_0^{+}, u_1^{+}, u_2^{+}\right)$ with
\begin{align*}
u_0^{+}=2, u_1^{+}=1+\cos t+\cos 5 t, u_2^{+}=\cos 2 t \cos 2 \tau,
\end{align*}
that gives a response
\begin{align*}
Y(\omega)=2+\pi^2\left(\xi_2(\omega)\right)^2+\pi\left(2 \xi_0(\omega)+\xi_1(\omega)+\xi_5(\omega)\right).
\end{align*}

We simulate $N$ independent samples  $(Y_i, X_i(\cdot))$ of $(Y(\omega), X(\omega, t))$ and use them to construct the regularized quadratic approximation
$u_{\lambda,q}^N= (u_{\lambda,0,q}^N,u_{\lambda,1,q}^N,u_{\lambda,2,q}^N)$ of $u^{+}=\left(u_0^{+}, u_1^{+}, u_2^{+}\right)$ by the iterated Tikhonov regularization with $q=4$, as it is described in the previous section.

On Figures \ref{fig1} - \ref{fig3} we plot the error $\norm{u^{+}-u_{\lambda, 4}^N}_{\CalL^2}$ against the number of the used samples $N=1,2, \ldots, 40$, for a  wide range of the regularization parameter, from $\lambda=10^{-1}$ to $\lambda=10^{-9}$.  For all that parameters we observe that the error exhibits a tendency of decreasing with the increase of the number of samples, and this is in agreement with our analysis.

Observe also that our toy problem is a finite dimensional one, such that $u^{+}=\left(u_0^{+}, u_1^{+}, u_2^{+}\right)$ is completely defined by at most 41 parameters. At the same time, as it can be seen from Figure \ref{fig3}, starting already from $N=27$ samples the regularized quadratic approximation
$u_{\lambda,4}^N= (u_{\lambda,0,4}^N,u_{\lambda,1,4}^N,u_{\lambda,2,4}^N)$ with $\lambda=10^{-9}$ gives a very accurate reconstruction of the exact $u^{+}=\left(u_0^{+}, u_1^{+}, u_2^{+}\right)$, such that

\begin{figure} 
\includegraphics[width=\textwidth]{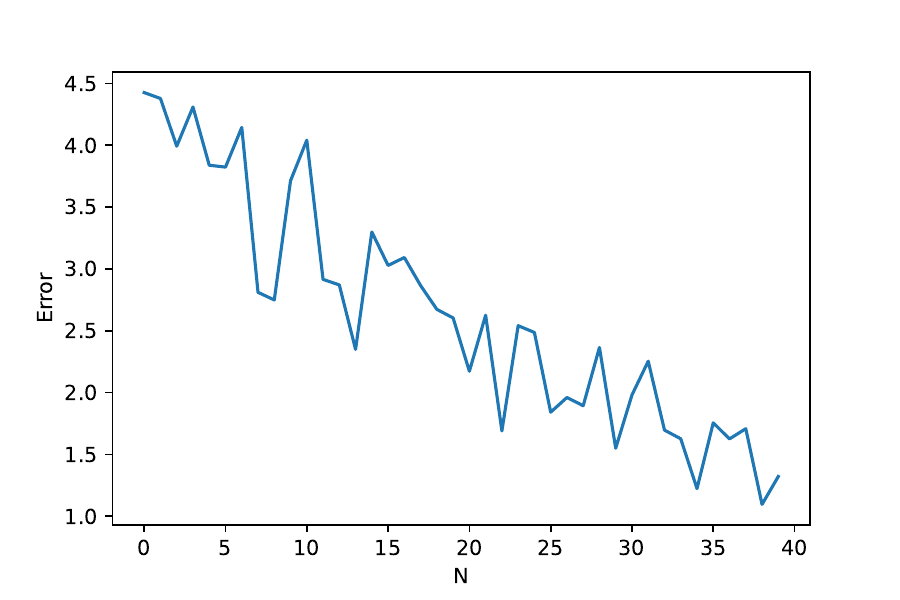}
\caption{Error curve for $\lambda=10^{-1}$}
\label{fig1}
\end{figure}

\begin{figure} 
\includegraphics[width=\textwidth]{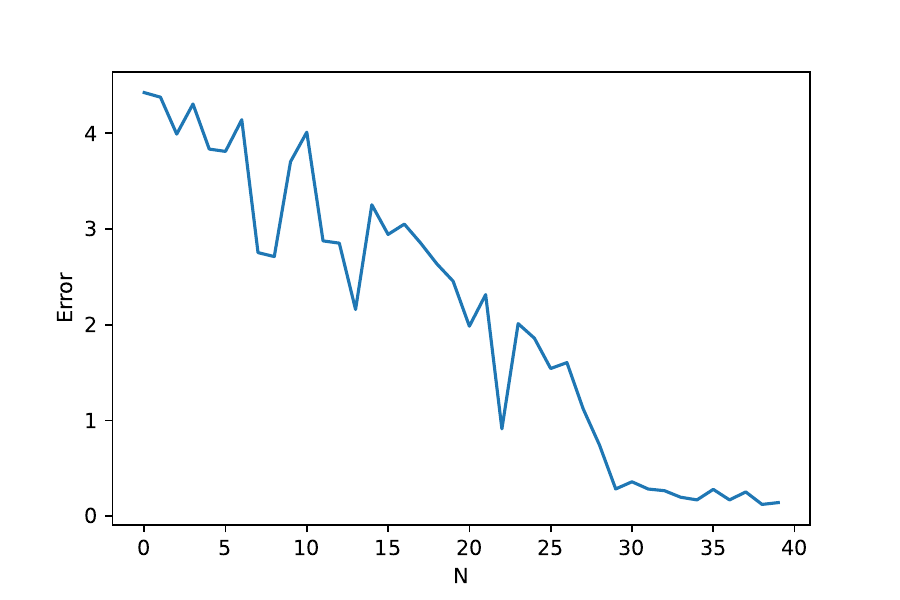}
\caption{Error curve for $\lambda=10^{-3}$}
\label{fig2}
\end{figure}

\begin{figure} 
\includegraphics[width=\textwidth]{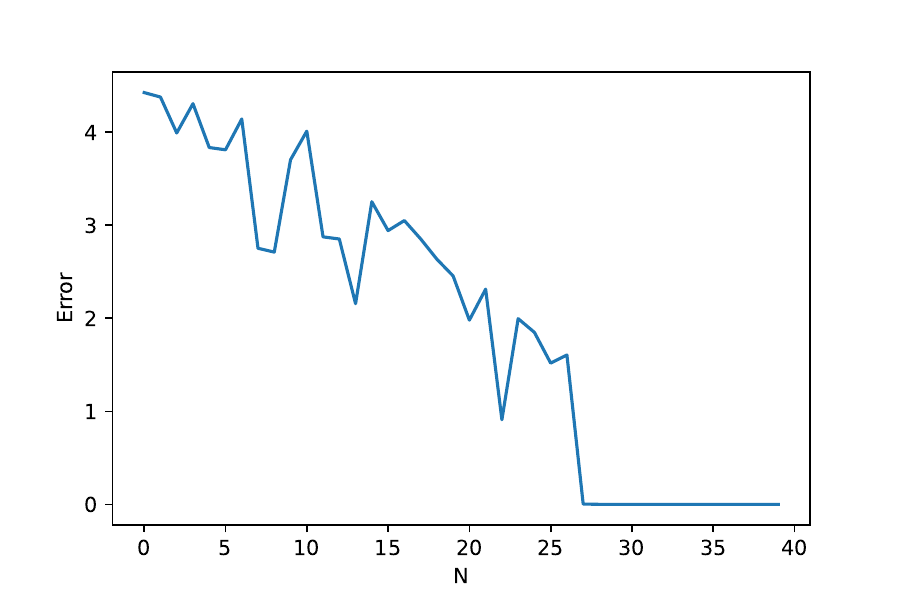}
\caption{Error curve for $\lambda=10^{-9}$}
\label{fig3}
\end{figure}

\begin{align*}
& u_{\lambda, 0,4}^N=1.99999, \quad u_{\lambda, 1,4}^N=1.0000+0.99999 \cos t+1.00000 \cos 5 t, \\
& u_{\lambda, 2,4}^N=0.99999 \cos 2 t \cos 2 \tau+\sum_{k \neq 2}^5 \beta_{k} \cos k t \cos k \tau, \\
&  \left|\beta_k\right|<7 \times 10^{-9}.
\end{align*}

In addition, we may report that in the considered example, the attempt to employ a linear functional regression model with $p=1$ gives misleading approximations (not exhibited in the figures).
At the same time, if we consider $u^{+}=\left(u_0^{+}, u_1^{+}\right)$, i.e. a linear response $Y(\omega)$, then both regression models (linear and quadratic) give good approximations starting from $N=27$ samples. It hints at a recommendation to not be afraid of the use of PFR of higher order $p>1$.

For completeness, let us also discuss the case, when we explicitly introduce noise to the labels. Specifically, let $\varepsilon \sim \mathcal{N}(0,10^{-3})$ additive Gaussian noise with mean $0$ and variance $10^{-3}$, so that the labels are now given as $\tilde{Y}(\omega)=Y(\omega)+\varepsilon(\omega)$, for $Y$ as in \eqref{eq:y_def_exp}. Then, in the experiments we overall observe similar behavior, the errors increase only slightly. Let us report the error curve corresponding to the case $\lambda=10^{-9}$ in Figure \ref{fig:noise}.

\begin{figure} 
\includegraphics[width=\textwidth]{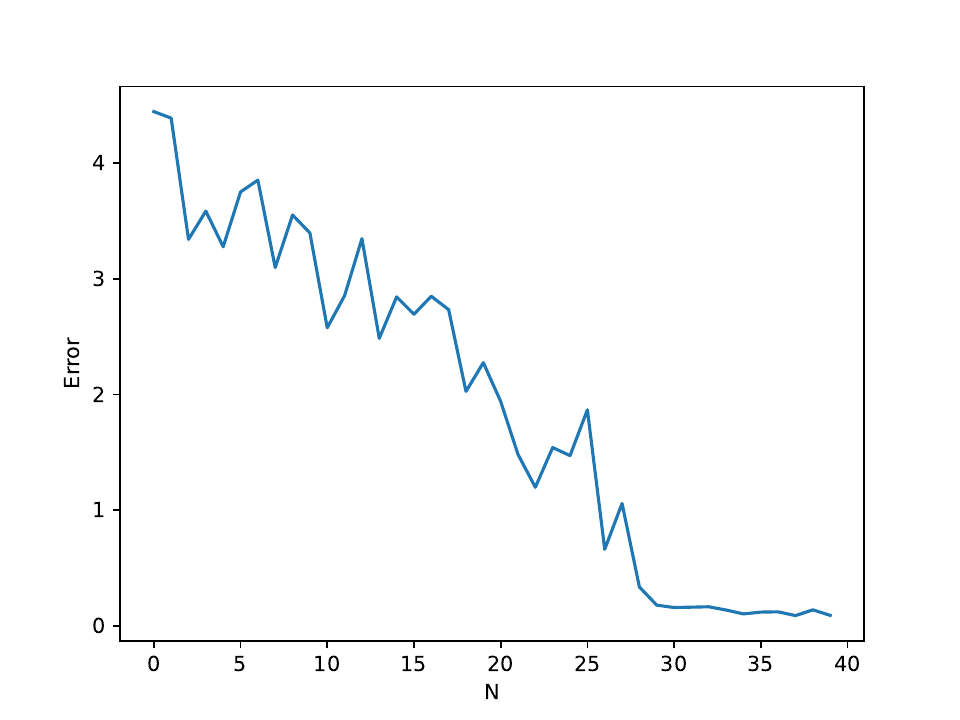}
\caption{Error curve for $\lambda=10^{-9}$ with additive Gaussian noise.}
\label{fig:noise}
\end{figure}

\section{Acknowledgements}
The authors are grateful to two anonymous referees for their comments and suggestions that led to improvements in this paper. 

The research reported in this paper has been supported by the Federal Ministry for Climate Action, Environment, Energy, Mobility, Innovation and Technology (BMK), the Federal Ministry for Digital and Economic Affairs (BMDW), and the Province
of Upper Austria in the frame of the COMET–Competence Centers for Excellent Technologies Programme and the COMET Module S3AI managed by the Austrian Research Promotion Agency FFG.


\bibliography{colt}
\bibliographystyle{plain}
\end{document}